\documentclass[11pt]{article}

\usepackage[a4paper,margin=1in]{geometry}
\usepackage{amsmath,amssymb,amsthm}
\usepackage{hyperref}
\usepackage{enumitem}

\newtheorem{theorem}{Theorem}[section]
\newtheorem{proposition}[theorem]{Proposition}

\newtheorem{corollary}[theorem]{Corollary}
\newtheorem{definition}[theorem]{Definition}

\title{Terminal Ambiguity Geometry for Finite-Dimensional Algebras}
\author{
Joe Gildea\\
Department of Computing Science and Mathematics,\\
School of Informatics and Creative Arts,\\
Dundalk Institute of Technology\\
\texttt{gildeajoe@gmail.com}}
\date{}

\begin{document}

\maketitle

\begin{abstract}

Radical truncation forgets information about a module. We study the ambiguity created by this loss: which nonisomorphic modules remain indistinguishable after all proper radical truncations have been fixed, how the resulting families are organized, and what information can be recovered from them. We introduce terminal ambiguity geometry to make this residual indeterminacy into a geometric object. The remaining ambiguity is governed by the terminal radical layer, and we classify terminal completions by Grassmannian orbit spaces, with projective ambiguity appearing as the rank-one case. This classification gives criteria for rigidity and nonuniqueness and, in the split case, determines the dimensions of the completion spaces and gives exact finite-field counts of nonisomorphic terminal completions.

We also establish reconstruction and symmetry results for the rank-one projective geometry. Under suitable rank hypotheses, the ambiguity geometry recovers the corresponding residue division-algebra data, terminal species, and terminal radical layer, while its automorphisms are described by classical projective semilinear groups. Thus the paper studies both sides of radical truncation: the ambiguity produced by the information it forgets and the extent to which that information can be recovered from the resulting geometry.
\end{abstract}



\section{Introduction}

The representation theory of finite-dimensional algebras has long
emphasized the role of radical filtrations, species, and
division-algebra valued invariants. Classical work of Gabriel, Dlab,
and Ringel showed that important structural information about an
algebra can be encoded in the interaction between its simple modules
and successive radical layers \cite{Gabriel72,DlabRingel76}. These
ideas have become fundamental in the modern representation theory of
Artin and finite-dimensional algebras \cite{ARS95,ASS06}.

Throughout this paper, let \(\Lambda\) be a finite-dimensional basic
algebra over a field \(k\), with Jacobson radical \(J\). For a
finite-dimensional \(\Lambda\)-module \(M\), the radical truncations
\[
M/JM,\;
M/J^2M,\;
\dots,\;
M/J^{n-1}M
\]
retain progressively more information about \(M\), while forgetting
its deeper radical structure. The perspective adopted here reverses
the usual question. Rather than asking what the radical filtration
records, we ask what information is lost under radical truncation and
what can be recovered from the resulting ambiguity.

For \(d\geq1\), write
\[
\tau_d(M)=M/J^dM.
\]
The ambiguity created by this truncation is measured by
\[
\mathcal F_d(M)
=
\Bigl\{
N\in\mathrm{mod}\text{-}\Lambda :
\tau_d(N)\cong\tau_d(M)
\Bigr\}/\cong.
\]
Thus \(\mathcal F_d(M)\) records the isomorphism classes of modules
that become indistinguishable after the \(d\)-th radical truncation.
The object of study is therefore the ambiguity itself: the information
about a module that is no longer determined once a radical truncation
has been fixed. In this sense, the ambiguity fibres
\(\mathcal F_d(M)\) are the basic objects of the theory developed
here. The central questions are what structure these fibres have, how
they can be classified and measured, and what information about the
original module-theoretic data can be recovered from them.

The present paper studies these questions at the terminal level, where
the ambiguity problem admits a particularly explicit solution.
Suppose \(J^n=0\). After the complete tower of proper radical
truncations has been fixed, the remaining ambiguity is concentrated
in the final nonzero radical layer \(J^{n-1}\). The relations
\(JJ^{n-1}=J^{n-1}J=0\) make this layer a semisimple
\((\Lambda/J)\)-bimodule, reducing the terminal ambiguity problem to
submodule geometry over residue division algebras. For primitive
idempotents \(e_i,e_j\), set
\[
V_{ij}=e_jJ^{n-1}e_i,
\qquad
D_h=e_h\Lambda e_h/e_hJe_h.
\]
Since \(\Lambda\) is basic, \(V_{ij}\) is naturally a
\((D_j,D_i)\)-bimodule. The left \(D_j\)-structure determines the
admissible terminal submodules, while the right \(D_i\)-structure
governs when two admissible choices yield isomorphic quotient
modules. This bimodule structure therefore supplies the algebraic
data governing the terminal ambiguity left by proper radical
truncation.

We introduce \emph{terminal ambiguity geometry} 
\(\mathfrak A^{\mathrm{top}}(\Lambda)\) as the structure carried by 
the ambiguity fibres that remain after all proper radical truncations 
have been fixed. The main classification theorem determines these 
fibres explicitly. For terminal quotient rank \(s\), one has 
\[ 
\mathcal F^{<n,\mathrm{top},s}_{ij} 
\cong 
\operatorname{Gr}_{D_j} 
(\ell_{ij}-s,V_{ij})/D_i^\times, 
\qquad 
\ell_{ij}=\dim_{D_j}V_{ij}. 
\] 
Thus the Grassmannian classification describes the terminal completions 
left unresolved by radical truncation. Varying \(s\) organizes these 
completions by terminal quotient rank, with projective ambiguity 
appearing as the rank-one case. The higher-rank classification shows 
that the information left unresolved by proper radical truncation is 
not exhausted by the rank-one projective fibre. 
 
This classification has a direct module-theoretic meaning. For an 
admissible left \(D_j\)-submodule \(U\subseteq V_{ij}\), the 
corresponding quotient \(M_U\) satisfies 
\[ 
M_U/J^dM_U 
\cong 
P_i/J^dP_i 
\qquad 
(1\leq d<n). 
\] 
Hence the Grassmannian points represent admissible terminal choices, 
while the \(D_i^\times\)-orbits represent isomorphism classes of 
modules having the same complete tower of proper radical truncations. 
The geometry therefore describes precisely the terminal information 
left unresolved by radical truncation. The classification gives a 
rigidity criterion: transitivity corresponds to a unique terminal 
completion, while failure of transitivity gives nonisomorphic 
completions with the same proper radical truncation data. 
 
In the split case, where \(D_i\cong D_j\cong k\), the classification 
becomes an ordinary Grassmannian: 
\[ 
\mathcal F^{<n,\mathrm{top},s}_{ij} 
\cong 
\operatorname{Gr}_k(m_{ij}-s,V_{ij}), 
\qquad 
m_{ij}=\dim_kV_{ij}. 
\] 
Their dimensions are 
\[ 
\dim 
\mathcal F^{<n,\mathrm{top},s}_{ij} 
= 
s(m_{ij}-s), 
\] 
so that
\[
\delta_{ij}(s)=s(m_{ij}-s)
\]
gives a dimension profile measuring the geometric size of the family
of rank-\(s\) terminal completions left undetermined by the same
proper radical truncation tower. This profile is maximal when \(s\)
is as close as possible to \(m_{ij}/2\), with
\[
\max_s\delta_{ij}(s)
=
\left\lfloor\frac{m_{ij}^2}{4}\right\rfloor.
\]
Over \(k=\mathbb F_q\), their cardinalities are 
\[ 
\left| 
\mathcal F^{<n,\mathrm{top},s}_{ij} 
\right| 
= 
{m_{ij}\brack s}_q. 
\] 
Thus terminal ambiguity can be measured geometrically and, over finite 
fields, the number of nonisomorphic terminal completions remaining 
indistinguishable by the same proper radical truncation data can be 
enumerated exactly. In particular, higher terminal ranks can exhibit 
more ambiguity than the rank-one projective fibre. 
 
The paper therefore has two complementary directions. The first 
determines and measures the ambiguity produced by radical truncation; 
the second asks how much of the information lost under truncation can 
be recovered from the structure of that ambiguity. Projective geometry 
provides classical reconstruction principles through the Fundamental 
Theorem of Projective Geometry and its generalizations over division 
rings \cite{Artin57,FaureFrolicher00}. Applied to the rank-one terminal 
ambiguity geometry, these principles lead, under the stated rank 
hypotheses, to reconstruction of the corresponding 
division-algebra-valued data, terminal species, and terminal radical 
layer. In the projective case, the automorphisms of an individual 
terminal ambiguity fibre are governed by projective semilinear 
groups. Thus the passage from radical truncation to ambiguity is not 
only a classification of what has been forgotten: in suitable cases, 
the geometry of that ambiguity allows terminal information to be 
recovered. 
 
Species enter this picture as algebraic data governing the terminal 
radical layer. Species theory records division algebras and bimodule 
structure, whereas terminal ambiguity geometry asks how that data is 
realized through modules that remain indistinguishable after proper 
radical truncation. The resulting ambiguity geometry therefore does 
not need to contain algebraic information independent of the full 
bimodule species. Its role is to turn the information lost under 
truncation into a concrete module-completion problem and to describe 
its classification, rigidity, nonuniqueness, dimension, enumeration, 
and symmetry. 
 
The algebraic background used throughout draws on standard results 
from the theory of Artin algebras, modules, radicals, and 
division-algebra valued structures 
\cite{ARS95,ASS06,Lam99,Pierce82,Jacobson85}. The species-theoretic 
viewpoint originates in the work of Gabriel and of Dlab and Ringel, 
where division-algebra valued data associated to radical layers was 
used to encode representation-theoretic structure 
\cite{Gabriel72,DlabRingel76}. The projective-geometric methods used 
here are inspired by the classical theory of projective spaces over 
division rings developed by Artin and by Faure and Fr\"olicher 
\cite{Artin57,FaureFrolicher00}. Recent work continues to illustrate 
the depth of the interaction between homological invariants and the 
structure of finite-dimensional algebras 
\cite{Ringel22,RingelZhang22}. Although the present work is not 
homological in nature, it similarly seeks to extract intrinsic 
structure from a distinguished layer of the algebra, namely the 
terminal radical layer. 
 
The paper is organized as follows. Section~2 introduces ambiguity 
fibres, radical truncation, terminal radical bimodules, and terminal 
ambiguity geometry. Section~3 develops the top-layer ambiguity 
construction and the isomorphism relation governing terminal 
completions. Section~4 develops projectivization over division 
algebras. Section~5 gives the rank-one classification of terminal 
ambiguity in projective terms. Section~6 develops the Grassmannian 
classification of higher-rank terminal ambiguity and, in the split 
case, measures this ambiguity through its dimension profile and 
finite-field enumeration. The subsequent sections develop the 
reconstruction theory, recover the terminal species and terminal 
radical layer under the stated hypotheses, and study the resulting 
symmetry theory.

\section{Terminal Ambiguity Fibres}

This section introduces the ambiguity fibres that motivate the geometric constructions developed later in the paper. We begin with the general ambiguity arising from radical truncation and then isolate the terminal ambiguity that survives after all proper truncations have been fixed. The central theme is that this residual ambiguity is governed entirely by the terminal radical layer \(J^{n-1}\). Throughout, let \(\Lambda\) be a finite-dimensional algebra over a field \(k\), and let \(J=\operatorname{rad}(\Lambda)\) denote its Jacobson radical. Since \(\Lambda\) is finite-dimensional, \(J\) is nilpotent, so there exists an integer \(n\ge1\) such that \(J^n=0\). We write \(\Lambda/J\) for the semisimple quotient algebra. Let \(e_1,\ldots,e_t\) be a complete set of primitive idempotents of \(\Lambda\). For each \(i\), set \(P_i=\Lambda e_i\) and \(D_i=e_i\Lambda e_i/e_iJe_i\). Then \(D_i\) is a division algebra and coincides with the endomorphism division algebra of the simple top of \(P_i\).

The final nonzero radical layer \(J^{n-1}\) plays a distinguished role throughout the paper. Since \(J^{n-1}J=J^n=0\), it is naturally a semisimple \((\Lambda/J)\)-bimodule. For primitive idempotents \(e_i,e_j\), define \(V_{ij}=e_jJ^{n-1}e_i\). Then \(V_{ij}\) is naturally a \((D_j,D_i)\)-bimodule \({}_{D_j}V_{ij}{}_{D_i}\). Let \(M\) be a finite-dimensional \(\Lambda\)-module. For each integer \(d\ge1\), define the \(d\)-th radical truncation by \(\tau_d(M)=M/J^dM\). When \(d<n\), the quotient \(\tau_d(M)\) forgets the deepest radical layers of \(M\). We shall refer to truncations with \(d<n\) as \emph{proper truncations}. The natural equivalence relation induced by radical truncation is the following. Two modules are regarded as indistinguishable at level \(d\) if their \(d\)-th radical truncations are isomorphic.

\begin{definition}
Let \(M\) be a finite-dimensional \(\Lambda\)-module and let \(d\ge1\). The ambiguity fibre of \(M\) at level \(d\) is
\[
\mathcal F_d(M)
=
\Bigl\{
N\in\mathrm{mod}\text{-}\Lambda :
\tau_d(N)\cong\tau_d(M)
\Bigr\}/\cong.
\]
Thus \(\mathcal F_d(M)\) records all isomorphism classes of modules having the same \(d\)-th radical truncation as \(M\).
\end{definition}

The original motivation for this work is to understand the geometry of these ambiguity fibres. The results of the present paper show that, at the terminal level, this ambiguity is controlled entirely by the final radical layer \(J^{n-1}\). This leads to a local ambiguity invariant attached to each nonzero component \(V_{ij}=e_jJ^{n-1}e_i\).

\begin{definition}
Fix primitive idempotents \(e_i,e_j\) and suppose \(V_{ij}=e_jJ^{n-1}e_i\neq0\). The terminal ambiguity fibre \(\mathcal F^{<n,\mathrm{top}}_{ij}\) is the ambiguity surviving after all proper truncations have been fixed and only the terminal layer remains undetermined.
\end{definition}

The key observation is that every proper truncation annihilates the final radical layer \(J^{n-1}\). Consequently, all ambiguity surviving proper truncation is concentrated in the semisimple bimodule \(J^{n-1}\). For a fixed pair \((i,j)\), the relevant local contribution is the component \(V_{ij}=e_jJ^{n-1}e_i\). Since \(V_{ij}J=0\), the natural right action of \(e_i\Lambda e_i\) on \(V_{ij}\) factors through the residue division algebra \(D_i=e_i\Lambda e_i/e_iJe_i\).

Thus terminal ambiguity reduces to a problem in linear algebra over division algebras. The remainder of the paper is devoted to determining the geometry of the fibres \(\mathcal F^{<n,\mathrm{top}}_{ij}\) and their relationship with the terminal species \(\mathsf{Sp}^{\mathrm{top}}(\Lambda)=\bigl(D_i,{}_{D_j}V_{ij}{}_{D_i}\bigr)_{i,j}\).

\section{Top-Layer Ambiguity}

The objective of this section is to determine the structure of the terminal ambiguity fibre introduced in the previous section. The first step is to identify the group acting on the terminal component \(V_{ij}=e_jJ^{n-1}e_i\). Although this action arises from the unit group of the corner algebra \(e_i\Lambda e_i\), the nilpotence of the radical implies that only the associated residue division algebra is visible on the terminal layer.

\begin{theorem}\label{thm:top-layer-factorization}
Let $\Lambda$ be a finite-dimensional algebra with radical $J$ satisfying $J^n=0$ for some \(n\ge 2\). Let \(e_i,e_j\) be primitive idempotents and set $V=e_jJ^{n-1}e_i$.
Then the action of $(e_i\Lambda e_i)^\times$ on \(V\) by right multiplication factors through $D_i^\times =(e_i\Lambda e_i/e_iJe_i)^\times$. Equivalently, the subgroup $1+e_iJe_i$
acts trivially on $V=e_jJ^{n-1}e_i$.
\end{theorem}

\begin{proof}
Let $u\in 1+e_iJe_i$. Then $u=e_i+r$ for some $r\in e_iJe_i$. Let $v\in V=e_jJ^{n-1}e_i$. Then $vu=v(e_i+r)=ve_i+vr$. Since $v\in e_jJ^{n-1}e_i$, we have $ve_i=v$. Moreover, $v\in J^{n-1}$ and $r\in J$. Therefore $vr\in J^n=0$. Hence $vu=v$. Thus every element of $1+e_iJe_i$
acts trivially on \(V\). Therefore the action of $(e_i\Lambda e_i)^\times$ on \(V\) has kernel containing \(1+e_iJe_i\), and hence factors through the quotient $(e_i\Lambda e_i)^\times/(1+e_iJe_i)$. Since \(e_i\) is primitive, the corner algebra $e_i\Lambda e_i$ is local, with radical $e_iJe_i$. Hence
\[
(e_i\Lambda e_i)^\times/(1+e_iJe_i)
\cong
(e_i\Lambda e_i/e_iJe_i)^\times
=
D_i^\times.
\]

Therefore the action factors through $D_i^\times$.
\end{proof}

The preceding result shows that the terminal layer remembers only the action of the residue division algebra \(D_i=e_i\Lambda e_i/e_iJe_i\). Consequently, the classification of terminal ambiguity fibres reduces to a projective-geometric quotient problem over a division algebra. The next theorem identifies the resulting ambiguity fibre explicitly.

\begin{theorem}\label{thm:top-layer-ambiguity}
Let \(\Lambda\) be a finite-dimensional basic split \(k\)-algebra with radical \(J\) satisfying \(J^n=0\). Let \(e_i,e_j\) be primitive idempotents and set
$P_i=\Lambda e_i$, $V_{ij}=e_jJ^{n-1}e_i$. Assume \(V_{ij}\neq0\). Then the intrinsic rank-one terminal ambiguity fibre associated to \(V_{ij}\) admits a natural description
$\mathcal F^{<n,\mathrm{top}}_{ij} \cong \mathbb P_k(V_{ij}^*)$.
\end{theorem}

\begin{proof}
Since \(J^n=0\), we have \(J(J^{n-1}P_i)=0\). Thus \(J^{n-1}P_i\) is semisimple. Let \(W_j=\bigoplus_{h\ne j}e_hJ^{n-1}P_i\). For each nonzero functional \(\lambda\in V_{ij}^*\), define \(K_\lambda=W_j\oplus\ker(\lambda)\subseteq J^{n-1}P_i\) and set \(M_\lambda=P_i/K_\lambda\). Since \(\Lambda\) is split, \(D_j\cong k\), so every \(k\)-subspace of \(V_{ij}\) is a left \(D_j\)-submodule. Hence \(\ker(\lambda)\), and therefore \(K_\lambda\), is a \(\Lambda\)-submodule of \(P_i\). Moreover, \(K_\lambda\subseteq J^{n-1}P_i\subseteq JP_i\), so the canonical surjection \(P_i\twoheadrightarrow M_\lambda\) is a projective cover. Now let \(1\le d<n\). Since \(K_\lambda\subseteq J^{n-1}P_i\subseteq J^dP_i\), we have \(J^dM_\lambda=J^dP_i/K_\lambda\). Therefore \(M_\lambda/J^dM_\lambda\cong(P_i/K_\lambda)/(J^dP_i/K_\lambda)\cong P_i/J^dP_i\). Thus the proper radical truncation tower of \(M_\lambda\) is independent of \(\lambda\). Hence all modules \(M_\lambda\) lie in the same proper truncation ambiguity fibre. It remains to determine when two such modules are isomorphic.

Suppose \(M_\lambda\cong M_\mu\). The construction depends only on the kernel of \(\lambda\). Indeed, if \(\mu=c\lambda\) for some \(c\in k^\times\), then \(\ker(\mu)=\ker(\lambda)\), and hence \(K_\mu=K_\lambda\) and \(M_\mu=M_\lambda\). Conversely, since \(P_i\to M_\lambda\) and \(P_i\to M_\mu\) are projective covers, an isomorphism \(M_\lambda\cong M_\mu\) lifts to an automorphism \(\varphi\in\operatorname{Aut}_\Lambda(P_i)\) satisfying \(\varphi(K_\lambda)=K_\mu\). By Theorem~\ref{thm:top-layer-factorization}, the induced action on \(V_{ij}\) factors through \(D_i^\times\). Since \(\Lambda\) is split, \(D_i^\times\cong k^\times\), and this action is scalar multiplication. Hence it fixes every \(k\)-subspace of \(V_{ij}\). Since \(\varphi(K_\lambda)=K_\mu\), comparison of the \(j\)-components gives \(\ker(\lambda)=\ker(\mu)\). Therefore \(\lambda\) and \(\mu\) differ by a nonzero scalar. It follows that the isomorphism classes in the intrinsic top-layer ambiguity fibre are naturally parametrised by \(\mathbb P_k(V_{ij}^*)\). Therefore \(\mathcal F^{<n,\mathrm{top}}_{ij}\cong\mathbb P_k(V_{ij}^*)\).
\end{proof}

The theorem gives the projective description of rank-one terminal ambiguity in the split case. For a general finite-dimensional algebra, the left residue division algebra \(D_j\) imposes additional restrictions on the admissible terminal submodules. The general nonsplit classification will be developed in Section~6 using the full bimodule structure of \(V_{ij}\).

\section{Projectivization over Division Algebras}

For a finite-dimensional right module \(V\) over a division algebra \(D\), consider the quotient
\[
\mathcal P_D(V)
=
\mathbb P_k(V^*)/D^\times.
\]
This quotient admits an intrinsic projective interpretation over \(D\). We develop that projectivization here as an independent geometric construction. Its relationship with terminal ambiguity in the general nonsplit setting will be clarified by the full bimodule classification in Section~6.

\begin{theorem}\label{thm:division-projectivization}
Let \(D\) be a finite-dimensional division algebra over \(k\), and let \(V\) be
a finite-dimensional right \(D\)-module. Set $D^*=\operatorname{Hom}_k(D,k)$, viewed as a \(D\)-bimodule by
$(a\cdot \eta)(x)=\eta(xa)$, $(\eta\cdot a)(x)=\eta(ax)$, for \(a,x\in D\) and \(\eta\in D^*\). Then there is a natural bijection
\[
\mathbb P_k(V^*)/D^\times
\cong
\mathbb P_D\!\left(\operatorname{Hom}_D(V,D^*)\right),
\]
where \(\mathbb P_D(-)\) denotes the set of one-dimensional left
\(D\)-submodules.
\end{theorem}

The proof proceeds in two stages. First we identify \(V^*\) with a naturally associated left
\(D\)-module. We then show that the quotient by \(D^\times\) becomes ordinary
projectivization over \(D\).

\begin{proof}
Define \(\Phi:V^*\longrightarrow \operatorname{Hom}_D(V,D^*)\) by
\(\Phi(\lambda)(v)(d)=\lambda(vd)\), for \(\lambda\in V^*\), \(v\in V\), and
\(d\in D\). We first show that \(\Phi(\lambda)\) is right \(D\)-linear. For \(a\in D\),
\(\Phi(\lambda)(va)(d)=\lambda(vad)\), while

\[(\Phi(\lambda)(v)\cdot a)(d)=\Phi(\lambda)(v)(ad)=\lambda(vad).\]
 Hence
\(\Phi(\lambda)(va)=\Phi(\lambda)(v)\cdot a\). Thus
\(\Phi(\lambda)\in \operatorname{Hom}_D(V,D^*)\) and the map \(\Phi\) is \(k\)-linear. To prove injectivity, suppose \(\Phi(\lambda)=0\). Then
\(\Phi(\lambda)(v)(1)=0\) for every \(v\in V\). Hence \(\lambda(v)=0\) for every \(v\in V\), so \(\lambda=0\). To prove surjectivity, let \(f\in \operatorname{Hom}_D(V,D^*)\). Define
\(\lambda_f(v)=f(v)(1)\). Then
\(\Phi(\lambda_f)(v)(d)=\lambda_f(vd)=f(vd)(1)\). Since \(f\) is right
\(D\)-linear, \(f(vd)=f(v)\cdot d\). Therefore
\[f(vd)(1)=(f(v)\cdot d)(1)=f(v)(d).\]
 Thus \(\Phi(\lambda_f)(v)=f(v)\) for every \(v\in V\), and hence \(\Phi(\lambda_f)=f\). Therefore \(V^*\cong\operatorname{Hom}_D(V,D^*)\) as \(k\)-vector spaces.

Set \(W=\operatorname{Hom}_D(V,D^*)\). The left action of \(D\) on \(D^*\) induces a left \(D\)-module structure on \(W\) by \((a\cdot f)(v)=a\cdot f(v)\). The action of \(D^\times\) on \(V^*\) is
\((a\cdot\lambda)(v)=\lambda(va^{-1})\). Under the isomorphism \(\Phi\), this action corresponds to left
multiplication by \(a^{-1}\) on \(W\). Indeed,
\[\Phi(a\cdot\lambda)(v)(d)=(a\cdot\lambda)(vd)=\lambda(vda^{-1}),\]
while
\[(a^{-1}\cdot\Phi(\lambda))(v)(d)=\Phi(\lambda)(v)(da^{-1})
=\lambda(vda^{-1}).\]
Hence \(\Phi(a\cdot\lambda)=a^{-1}\cdot\Phi(\lambda)\). Therefore
\(\mathbb P_k(V^*)/D^\times\cong\mathbb P_k(W)/D^\times\). Since the \(D^\times\)-orbits on nonzero elements of \(W\) are precisely the
nonzero elements lying in the same one-dimensional left \(D\)-submodule,
we have $\mathbb P_k(W)/D^\times \cong \mathbb P_D(W)$. Substituting \(W=\operatorname{Hom}_D(V,D^*)\) gives
\[
\mathbb P_k(V^*)/D^\times
\cong
\mathbb P_D\!\left(
\operatorname{Hom}_D(V,D^*)
\right).
\]
\end{proof}

The preceding theorem identifies the ambiguity quotient with a genuine
projective geometry over \(D\). The remaining question is to determine
which projective geometry it is. Since every finite-dimensional module over a division algebra is free,
one expects the resulting projective geometry to depend only on the
\(D\)-rank of \(V\). The next theorem shows that this is indeed the
case.

\begin{theorem}\label{thm:nonsplit-projective-space}
Let \(D\) be a finite-dimensional division algebra over \(k\), and let
\(V\) be a finite-dimensional right \(D\)-module. Let $r=\dim_D V$ and
let $D^*=\operatorname{Hom}_k(D,k)$, viewed as a \(D\)-bimodule by $(a\cdot \eta)(x)=\eta(xa)$, $(\eta\cdot a)(x)=\eta(ax)$, for \(a,x\in D\) and \(\eta\in D^*\). Then there is a natural identification
\[
\mathcal P_D(V)
=
\mathbb P_k(V^*)/D^\times
\cong
\mathbb P_D\!\left(\operatorname{Hom}_D(V,D^*)\right),
\]
where \(\mathbb P_D(-)\) denotes the set of one-dimensional left
\(D\)-submodules. In particular, since \(V\cong D^r\), there is a noncanonical identification
\[
\mathcal P_D(V)
\cong
\mathbb P_D^{\,r-1}.
\]
\end{theorem}

The first theorem already identifies
\(\mathcal P_D(V)\) with a projective space over \(D\).
The remaining task is therefore to compute the corresponding
\(D\)-dimension.

\begin{proof}
By Theorem~\ref{thm:division-projectivization},
\[
\mathcal P_D(V)
\cong
\mathbb P_D\!\left(
\operatorname{Hom}_D(V,D^*)
\right).
\]

It therefore suffices to determine the left \(D\)-module structure of
\(\operatorname{Hom}_D(V,D^*)\). Since \(V\) is a finite-dimensional right module over the division algebra
\(D\), it is free. Hence \(V\cong D^r\) as a right \(D\)-module, where
\(r=\dim_D V\). There is a natural \(k\)-linear isomorphism
\(\Phi:V^*\longrightarrow \operatorname{Hom}_D(V,D^*)\) defined by
\(\Phi(\lambda)(v)(d)=\lambda(vd)\), for \(\lambda\in V^*\), \(v\in V\), and
\(d\in D\). We first check that \(\Phi(\lambda)\) is right \(D\)-linear. For \(a\in D\),
\(\Phi(\lambda)(va)(d)=\lambda(vad)\). On the other hand,
\((\Phi(\lambda)(v)\cdot a)(d)=\Phi(\lambda)(v)(ad)=\lambda(vad)\). Hence
\(\Phi(\lambda)(va)=\Phi(\lambda)(v)\cdot a\), so
\(\Phi(\lambda)\in \operatorname{Hom}_D(V,D^*)\).

The map \(\Phi\) is injective because
\(\Phi(\lambda)(v)(1)=\lambda(v)\). It is also surjective: if
\(f\in \operatorname{Hom}_D(V,D^*)\), define
\(\lambda_f(v)=f(v)(1)\). Then
\(\Phi(\lambda_f)(v)(d)=\lambda_f(vd)=f(vd)(1)\). Since \(f\) is right
\(D\)-linear, \(f(vd)=f(v)\cdot d\). Therefore
\(f(vd)(1)=(f(v)\cdot d)(1)=f(v)(d)\). Thus
\(\Phi(\lambda_f)=f\), and \(\Phi\) is an isomorphism.

Now set \(W=\operatorname{Hom}_D(V,D^*)\). The left action of \(D\) on \(D^*\)
induces a left \(D\)-module structure on \(W\) by
\((a\cdot f)(v)=a\cdot f(v)\). Under the isomorphism \(\Phi\), the induced \(D^\times\)-action on
\(\mathbb P_k(V^*)\) corresponds to multiplication by nonzero left
\(D\)-scalars on \(W\). Hence
\(\mathbb P_k(V^*)/D^\times\cong\mathbb P_D(W)\). It remains to identify \(W\). Since \(V\cong D^r\), we have
\[
W=\operatorname{Hom}_D(V,D^*)\cong\operatorname{Hom}_D(D^r,D^*)
\cong(D^*)^r\] 
as left \(D\)-modules. As a left \(D\)-module, \(D^*\) is one-dimensional. Indeed,
\(\dim_k D^*=\dim_k D\), and every finite-dimensional left \(D\)-module is
free. Therefore \(D^*\cong D\) as a left \(D\)-module, though not canonically. Consequently, \(W\cong D^r\) as a left \(D\)-module. Hence
\(\mathbb P_D(W)\cong\mathbb P_D(D^r)=\mathbb P_D^{\,r-1}\).

Combining the identifications gives
\(\mathcal P_D(V)=\mathbb P_k(V^*)/D^\times
\cong\mathbb P_D\!\left(\operatorname{Hom}_D(V,D^*)\right)
\cong\mathbb P_D^{\,r-1}\). Thus the ambiguity quotient of a right \(D\)-module of rank \(r\) is,
noncanonically, projective \((r-1)\)-space over \(D\).
\end{proof}

Thus, for a finite-dimensional right \(D\)-module \(V\), the quotient
\[
\mathcal P_D(V)=\mathbb P_k(V^*)/D^\times
\]
is intrinsically a projective space over \(D\), noncanonically determined by the \(D\)-rank of \(V\). The application of this construction to terminal ambiguity requires additional compatibility with the left residue-division-algebra action, which will be incorporated in Section~6.

\section{Complete Classification of Terminal Ambiguity Fibres}

We now complete the rank-one projective classification in the split case.
Here the residue division algebras coincide with the ground field, so the
terminal hyperplanes are ordinary \(k\)-hyperplanes and the ambiguity fibre
is an ordinary projective space. The general nonsplit situation requires
the full bimodule structure of the terminal component and will be treated
in Section~6.

\begin{theorem}\label{thm:complete-classification}
Let \(\Lambda\) be a finite-dimensional basic split \(k\)-algebra with radical \(J\) satisfying \(J^n=0\). Let \(e_i,e_j\) be primitive idempotents and set \(V_{ij}=e_jJ^{n-1}e_i\). Assume \(V_{ij}\neq0\), and set \(m_{ij}=\dim_kV_{ij}\). Then there is a natural bijection \(\mathcal F^{<n,\mathrm{top}}_{ij}\cong\mathbb P_k(V_{ij}^*)\). Consequently, after choosing a basis of \(V_{ij}\), there is a noncanonical bijection \(\mathcal F^{<n,\mathrm{top}}_{ij}\cong\mathbb P_k^{\,m_{ij}-1}\).
\end{theorem}

\begin{proof}
By Theorem~\ref{thm:top-layer-ambiguity}, \(\mathcal F^{<n,\mathrm{top}}_{ij}\cong\mathbb P_k(V_{ij}^*)\). Since \(\dim_kV_{ij}=m_{ij}\), choosing a basis of \(V_{ij}\) gives a noncanonical identification \(\mathbb P_k(V_{ij}^*)\cong\mathbb P_k^{\,m_{ij}-1}\).
\end{proof}

Thus, in the split case, the rank-one terminal ambiguity fibre is completely
determined by the multiplicity \(m_{ij}=\dim_kV_{ij}\).

The split classification shows that rank-one terminal ambiguity is concentrated in
the final radical layer and is measured entirely by the multiplicity of the
component \(e_jJ^{n-1}e_i\). In particular, whenever \(\Lambda\) is split, the
geometry of the rank-one ambiguity fibre is completely determined by the single integer
\(m_{ij}=\dim_k(e_jJ^{n-1}e_i)\). Thus the family of rank-one terminal ambiguity fibres
ranges from a single point when \(m_{ij}=1\) to higher-dimensional projective
spaces as the multiplicity increases.

\section{Grassmannian Classification of Terminal Ambiguity}

This section extends terminal ambiguity from rank-one terminal quotients
to arbitrary terminal rank. The first subsection defines the higher-rank
ambiguity fibres and classifies the terminal completions that remain
indistinguishable under every proper radical truncation. The second
specializes to the split case, where the resulting ambiguity can be
measured geometrically by dimension and, over finite fields, by exact
enumeration.

\subsection{Higher-rank terminal ambiguity and classification}

We enlarge the terminal ambiguity problem from one-dimensional terminal
quotients to quotients of arbitrary rank. Fix primitive idempotents
\(e_i,e_j\) and suppose $V_{ij}=e_jJ^{n-1}e_i\neq0$. Write $\ell_{ij}=\dim_{D_j}V_{ij}$,
\[
W_j=\bigoplus_{h\neq j}e_hJ^{n-1}P_i.
\]
For \(1\leq s\leq\ell_{ij}\), let $\operatorname{Gr}_{D_j}(\ell_{ij}-s,V_{ij})$ denote the Grassmannian of left \(D_j\)-submodules \(U\subseteq V_{ij}\) satisfying
$\dim_{D_j}(V_{ij}/U)=s$. For such \(U\), set $K_U=W_j\oplus U$, $M_U=P_i/K_U$.

\begin{definition}
The rank-\(s\) terminal ambiguity fibre associated to \(V_{ij}\) is
\[
\mathcal F^{<n,\mathrm{top},s}_{ij}
=
\left\{
M_U:
U\in
\operatorname{Gr}_{D_j}(\ell_{ij}-s,V_{ij})
\right\}/\cong.
\]
\end{definition}

The first point is that varying \(U\) changes only the terminal
completion: it does not alter any proper radical truncation.

\begin{proposition}\label{prop:higher-rank-truncation}
For every $U\in \operatorname{Gr}_{D_j}(\ell_{ij}-s,V_{ij})$, the subspace $K_U=W_j\oplus U$ is a left \(\Lambda\)-submodule of \(P_i\). Moreover, for every
\(1\leq d<n\), $M_U/J^dM_U \cong P_i/J^dP_i$. Thus all modules \(M_U\) have the same complete tower of proper radical truncations.
\end{proposition}

\begin{proof}
Since \(JV_{ij}=0\), the left \(\Lambda\)-action on \(V_{ij}\) factors through the \(j\)-component \(D_j\) of \(\Lambda/J\). Hence every left \(D_j\)-submodule \(U\subseteq V_{ij}\) is a left \(\Lambda\)-submodule. Since
\[
W_j=
\bigoplus_{h\neq j}e_hJ^{n-1}P_i
\]
is a direct sum of \(\Lambda/J\)-isotypic components, it is also a left \(\Lambda\)-submodule. Hence $K_U=W_j\oplus U$ is a left \(\Lambda\)-submodule of \(P_i\). Since
$K_U\subseteq J^{n-1}P_i\subseteq JP_i$, the canonical surjection $P_i\twoheadrightarrow M_U$ is a projective cover. For \(1\leq d<n\), one has $K_U\subseteq J^{n-1}P_i\subseteq J^dP_i$,
and therefore $J^dM_U=J^dP_i/K_U$. Consequently,
\[
M_U/J^dM_U
\cong
(P_i/K_U)/(J^dP_i/K_U)
\cong
P_i/J^dP_i.
\]
\end{proof}

We now classify these terminal completions up to isomorphism. The two
sides of the terminal bimodule ${}_{D_j}V_{ij}{}_{D_i}$ play different roles: the left \(D_j\)-module structure determines the
admissible terminal submodules \(U\), while the right \(D_i\)-action
determines when two such submodules yield isomorphic quotient modules.
Since the two actions commute, \(D_i^\times\) acts by right
multiplication on $\operatorname{Gr}_{D_j}(\ell_{ij}-s,V_{ij})$. The ambiguity fibre is therefore naturally governed by the orbit
structure of this action.

\begin{theorem}\label{thm:grassmannian-classification}
Let \(\Lambda\) be a finite-dimensional basic algebra with radical \(J\) satisfying \(J^n=0\). Let \(e_i,e_j\) be primitive idempotents, and set $P_i=\Lambda e_i$, $D_h=e_h\Lambda e_h/e_hJe_h$, $V_{ij}=e_jJ^{n-1}e_i$. Assume \(V_{ij}\neq0\), and write $\ell_{ij}=\dim_{D_j}V_{ij}$. Then, for every \(1\leq s\leq\ell_{ij}\), there is a natural bijection
\[
\mathcal F^{<n,\mathrm{top},s}_{ij}
\cong
\operatorname{Gr}_{D_j}
(\ell_{ij}-s,V_{ij})/D_i^\times,
\]
where \(D_i^\times\) acts by right multiplication. More precisely, if $U,U' \in \operatorname{Gr}_{D_j} (\ell_{ij}-s,V_{ij})$, then $M_U\cong M_{U'}$ if and only if $U'=Ua$
for some \(a\in D_i^\times\).
\end{theorem}

\begin{proof}
Set
\[
W_j=
\bigoplus_{h\neq j}e_hJ^{n-1}P_i,
\qquad
K_U=W_j\oplus U,
\qquad
M_U=P_i/K_U.
\]
By Proposition~\ref{prop:higher-rank-truncation}, \(K_U\) is a \(\Lambda\)-submodule of \(P_i\), and $P_i\twoheadrightarrow M_U$ is a projective cover. Suppose first that
$M_U\cong M_{U'}$. Since both canonical maps $P_i\twoheadrightarrow M_U$, $P_i\twoheadrightarrow M_{U'}$ are projective covers, an isomorphism $M_U\xrightarrow{\sim}M_{U'}$
lifts, by uniqueness of projective covers, to an automorphism $\varphi\in\operatorname{Aut}_{\Lambda}(P_i)$ satisfying $\varphi(K_U)=K_{U'}$. There is a natural identification
$\operatorname{Aut}_{\Lambda}(P_i) \cong (e_i\Lambda e_i)^\times$, under which an automorphism acts on \(P_i=\Lambda e_i\) by right multiplication. By
Theorem~\ref{thm:top-layer-factorization}, the induced action on $V_{ij}=e_jJ^{n-1}e_i$ factors through $D_i^\times = (e_i\Lambda e_i/e_iJe_i)^\times$. Thus there exists \(a\in D_i^\times\) such that the restriction of \(\varphi\) to \(V_{ij}\) is right multiplication by \(a\). Right multiplication preserves each left vertex component $e_hJ^{n-1}P_i$. Indeed, if \(x\in e_hJ^{n-1}P_i\), then $e_h(xa)=(e_hx)a=xa$. Hence \(\varphi\) preserves
\[
W_j=
\bigoplus_{h\neq j}e_hJ^{n-1}P_i.
\]
It follows that $\varphi(K_U) = \varphi(W_j\oplus U) = W_j\oplus Ua$. Since $\varphi(K_U)=K_{U'}=W_j\oplus U'$, comparison of the \(j\)-components gives $U'=Ua$. Conversely, suppose
$U'=Ua$ for some \(a\in D_i^\times\). Since $e_i\Lambda e_i$ is a finite-dimensional local algebra with radical \(e_iJe_i\), the canonical homomorphism
$(e_i\Lambda e_i)^\times \longrightarrow D_i^\times$ is surjective. Choose $\widetilde a\in(e_i\Lambda e_i)^\times$ whose image in \(D_i^\times\) is \(a\). Right multiplication by \(\widetilde a\) defines an automorphism $\varphi_{\widetilde a} \in \operatorname{Aut}_{\Lambda}(P_i)$. On \(J^{n-1}P_i\), its action depends only on the image \(a\), since the subgroup $1+e_iJe_i$ acts trivially on the terminal layer. Consequently, $\varphi_{\widetilde a}(W_j)=W_j$ and $\varphi_{\widetilde a}(U)=Ua=U'$. Hence $\varphi_{\widetilde a}(K_U)=K_{U'}$. The automorphism therefore descends to an isomorphism $P_i/K_U \xrightarrow{\sim} P_i/K_{U'}$, that is, $M_U\cong M_{U'}$. It follows that the isomorphism classes of rank-\(s\) terminal completions are precisely the \(D_i^\times\)-orbits on $\operatorname{Gr}_{D_j} (\ell_{ij}-s,V_{ij})$, and hence $\mathcal F^{<n,\mathrm{top},s}_{ij} \cong \operatorname{Gr}_{D_j} (\ell_{ij}-s,V_{ij})/D_i^\times$.
\end{proof}

Theorem~\ref{thm:grassmannian-classification} gives an immediate
rigidity criterion. The rank-\(s\) terminal completion is unique up to
isomorphism precisely when \(D_i^\times\) acts transitively on
$\operatorname{Gr}_{D_j}(\ell_{ij}-s,V_{ij})$.
Failure of transitivity is therefore equivalent to the existence of
nonisomorphic rank-\(s\) terminal completions having the same complete
tower of proper radical truncations. Thus the orbit structure measures
the failure of the truncation data to determine a unique terminal
completion. The case \(s=1\) recovers the rank-one terminal ambiguity of
Section~5. Thus the higher-rank classification shows that the
information left unresolved by proper radical truncation is not
exhausted by the rank-one projective fibre.

\subsection{Measurement of terminal ambiguity in the split case}

The classification becomes particularly explicit when \(\Lambda\) is
split over \(k\). In this case the residue division algebras satisfy
\(D_h\cong k\), and the right unit action reduces to scalar
multiplication. The orbit spaces of the general classification
therefore become ordinary Grassmannians, allowing the terminal
ambiguity to be measured by familiar geometric invariants.

\begin{theorem}\label{thm:split-grassmannian-classification}
Let \(\Lambda\) be a finite-dimensional basic split \(k\)-algebra with radical \(J\) satisfying \(J^n=0\). Let \(e_i,e_j\) be primitive idempotents and suppose
$V_{ij}=e_jJ^{n-1}e_i\neq0$. Set $m_{ij}=\dim_kV_{ij}$. Then, for every $1\leq s\leq m_{ij}$, there is a natural bijection
\[
\mathcal F^{<n,\mathrm{top},s}_{ij}
\cong
\operatorname{Gr}_k(m_{ij}-s,V_{ij}).
\]
Thus the rank-\(s\) terminal ambiguity fibre is identified with the
ordinary Grassmannian of codimension-\(s\) \(k\)-subspaces of
\(V_{ij}\).
\end{theorem}

\begin{proof}
By
Theorem~\ref{thm:grassmannian-classification},
\[
\mathcal F^{<n,\mathrm{top},s}_{ij}
\cong
\operatorname{Gr}_{D_j}
(m_{ij}-s,V_{ij})/D_i^\times.
\]
Since \(\Lambda\) is split, $D_j\cong k$ and $D_i\cong k$. Hence $\operatorname{Gr}_{D_j} (m_{ij}-s,V_{ij}) = \operatorname{Gr}_k (m_{ij}-s,V_{ij})$,
while $D_i^\times\cong k^\times$. The action of \(k^\times\) on \(V_{ij}\) is scalar multiplication. Every \(k\)-subspace of \(V_{ij}\) is fixed by scalar multiplication,
so the induced action of \(k^\times\) on $\operatorname{Gr}_k(m_{ij}-s,V_{ij})$ is trivial. Consequently,
$\operatorname{Gr}_k(m_{ij}-s,V_{ij})/k^\times = \operatorname{Gr}_k(m_{ij}-s,V_{ij})$,
and therefore
\[
\mathcal F^{<n,\mathrm{top},s}_{ij}
\cong
\operatorname{Gr}_k(m_{ij}-s,V_{ij}).
\]
\end{proof}

\begin{corollary}\label{cor:split-ambiguity-dimension}
Under the hypotheses of
Theorem~\ref{thm:split-grassmannian-classification}, the rank-\(s\)
terminal ambiguity fibre carries the structure of a Grassmannian
variety of dimension $\dim \mathcal F^{<n,\mathrm{top},s}_{ij} = s(m_{ij}-s)$.
\end{corollary}

\begin{proof}
By Theorem~\ref{thm:split-grassmannian-classification}, the ambiguity fibre is identified with $\operatorname{Gr}_k(m_{ij}-s,V_{ij})$. The Grassmannian of \(r\)-dimensional subspaces of an
\(m\)-dimensional vector space is a projective variety of dimension $r(m-r)$. Taking $r=m_{ij}-s$ gives $\dim \mathcal F^{<n,\mathrm{top},s}_{ij} = (m_{ij}-s)s =s(m_{ij}-s)$.
\end{proof}

The formula $\delta_{ij}(s)=s(m_{ij}-s)$, $1\leq s\leq m_{ij}$,
therefore gives a dimension profile for terminal ambiguity: it measures
the geometric size of the family of rank-\(s\) terminal completions left
undetermined by the same proper radical truncation tower. This profile
is maximal when \(s\) is as close as possible to \(m_{ij}/2\), with
\[
\max_s\delta_{ij}(s)
=
\left\lfloor\frac{m_{ij}^{2}}{4}\right\rfloor.
\]
In particular, for \(m_{ij}\geq4\), the rank-one fibre need not be the
largest ambiguity stratum. Over a finite field, the same classification gives a complementary
discrete measure of ambiguity by counting the corresponding isomorphism
classes exactly.

\begin{theorem}\label{thm:finite-field-terminal-count}
Assume the hypotheses of Theorem~\ref{thm:split-grassmannian-classification}, and suppose in addition that $k=\mathbb F_q$. Then, for every
$1\leq s\leq m_{ij}$, the number of isomorphism classes of rank-\(s\) terminal completions
with the prescribed proper radical truncation tower is
\[
\left|
\mathcal F^{<n,\mathrm{top},s}_{ij}
\right|
=
{m_{ij}\brack s}_q,
\]
where
\[
{m\brack s}_q
=
\prod_{r=0}^{s-1}
\frac{q^{m-r}-1}{q^{s-r}-1}
\]
is the Gaussian binomial coefficient.
\end{theorem}

\begin{proof}
By Theorem~\ref{thm:split-grassmannian-classification}, $\mathcal F^{<n,\mathrm{top},s}_{ij} \cong \operatorname{Gr}_{\mathbb F_q} (m_{ij}-s,V_{ij})$. The number of \((m_{ij}-s)\)-dimensional subspaces of an
\(m_{ij}\)-dimensional vector space over \(\mathbb F_q\) is
\[
{m_{ij}\brack m_{ij}-s}_q.
\]
Using the symmetry of Gaussian binomial coefficients,
\[
{m_{ij}\brack m_{ij}-s}_q
=
{m_{ij}\brack s}_q,
\]
and hence
\[
\left|
\mathcal F^{<n,\mathrm{top},s}_{ij}
\right|
=
{m_{ij}\brack s}_q.
\]
\end{proof}

The higher-rank classification therefore describes terminal ambiguity
beyond the rank-one projective picture. In the split case, the
dimension profile \(s\mapsto s(m_{ij}-s)\) measures the geometric size
of the ambiguity, while over finite fields the Gaussian binomial
coefficient \({m_{ij}\brack s}_q\) gives the exact number of
nonisomorphic terminal completions at each rank. Thus these
Grassmannian parameter spaces describe and measure the terminal
information left unresolved by proper radical truncation.

\section{Reconstruction from Ambiguity Geometry}

The preceding classification shows that rank-one terminal ambiguity is generally an orbit space of left \(D_j\)-hyperplanes under the right action of \(D_i^\times\). When \(D_j\cong k\), this reduces to the projective quotient studied in Section~4, and hence to a projective space over \(D_i\). We now study reconstruction in this projective case.

The projective geometry of such a terminal ambiguity fibre determines the rank of the corresponding terminal component and, in sufficiently large dimension, the residue division algebra \(D_i\). This provides the reconstruction results used below.

We begin with the simplest invariant. The dimension of a projective space
should be recoverable from its incidence geometry. For projective spaces
over division algebras this remains true and immediately yields recovery
of the terminal rank.

\begin{theorem}\label{thm:rank-recovery}
Let \(D\) be a division algebra, and let \(r,s\ge 1\). Suppose there is an
isomorphism of projective geometries $\mathbb P_D^{\,r-1}\cong\mathbb P_D^{\,s-1}$,
meaning a bijection preserving projective linear subspaces. Then $r=s$.
\end{theorem}

\begin{proof}
By definition, \(\mathbb P_D^{\,r-1}=\mathbb P_D(D^r)\) is the projective
geometry whose points are one-dimensional left \(D\)-submodules of \(D^r\). The projective dimension of \(\mathbb P_D(D^r)\) is \(r-1\). Equivalently, the
largest possible length of a strictly increasing chain of projective linear
subspaces is \(r-1\). Indeed, projective linear subspaces of \(\mathbb P_D(D^r)\) are precisely the
subsets of the form \(\mathbb P_D(U)\), where \(U\subseteq D^r\) is a nonzero
left \(D\)-submodule. Thus a strictly increasing chain of projective linear subspaces
\[\mathbb P_D(U_1)\subsetneq\mathbb P_D(U_2)\subsetneq\cdots\subsetneq
\mathbb P_D(U_m)\subseteq\mathbb P_D(D^r)\] corresponds exactly to a strictly
increasing chain of nonzero left \(D\)-submodules
\[U_1\subsetneq U_2\subsetneq\cdots\subsetneq U_m\subseteq D^r.\]

Since every left \(D\)-submodule of \(D^r\) is free, the ranks strictly increase
along such a chain. Hence \(m\le r\). Moreover, the bound is attained by
choosing a basis \(e_1,\ldots,e_r\) of \(D^r\) and taking
\(U_t=De_1\oplus\cdots\oplus De_t\) for \(1\le t\le r\). Thus the maximal
number of terms in such a chain is \(r\), and the projective dimension is
\(r-1\). If \(\mathbb P_D^{\,r-1}\cong\mathbb P_D^{\,s-1}\) as projective geometries,
then projective linear subspaces, and hence maximal chains of projective linear
subspaces, are preserved. Therefore the two projective geometries have the same
projective dimension: \(r-1=s-1\). Thus \(r=s\).
\end{proof}

Thus the rank of the underlying module is encoded intrinsically in the
projective geometry. Applied to terminal ambiguity fibres, this shows that
the size of the terminal component \(V_{ij}\) can be read directly from the
geometry of the fibre.

\begin{corollary}\label{cor:terminal-rank-recovery}
Let \(\Lambda\) be a finite-dimensional algebra with radical \(J\) satisfying \(J^n=0\). Let \(e_i,e_j\) be primitive idempotents and set \(D_i=e_i\Lambda e_i/e_iJe_i\), \(D_j=e_j\Lambda e_j/e_jJe_j\), \(V_{ij}=e_jJ^{n-1}e_i\), and \(r_{ij}=\dim_{D_i}V_{ij}\). Assume \(V_{ij}\neq0\) and \(D_j\cong k\). Then the projective dimension of the rank-one terminal ambiguity fibre recovers the terminal right rank:
\[
r_{ij}
=
\dim \mathcal F^{<n,\mathrm{top},1}_{ij}+1.
\]
\end{corollary}

\begin{proof}
Since \(D_j\cong k\), Theorem~\ref{thm:grassmannian-classification}, specialized to \(s=1\), identifies the rank-one terminal ambiguity fibre with the quotient \(\mathbb P_k(V_{ij}^*)/D_i^\times\). By Theorem~\ref{thm:nonsplit-projective-space}, this quotient is noncanonically isomorphic to \(\mathbb P_{D_i}^{\,r_{ij}-1}\). Therefore its projective dimension is \(r_{ij}-1\).
\end{proof}

Recovering rank is only the first step. A more substantial question is
whether the projective geometry also remembers the division algebra over
which it is defined. In dimensions at least two, the Fundamental Theorem
of Projective Geometry shows that this is indeed the case.

\begin{theorem}\label{thm:division-algebra-recovery}
Let \(D\) and \(E\) be division algebras, and let \(r,s\ge 3\). Suppose there
is an isomorphism of projective geometries
\[
\mathbb P_D^{\,r-1}
\cong
\mathbb P_E^{\,s-1},
\]
meaning a bijection preserving projective linear subspaces. Then $r=s$. Moreover, the division algebras \(D\) and \(E\) are isomorphic, up to the
left/right convention for projective spaces. Equivalently, one obtains $D\cong E$ or, if opposite-side conventions are used,
\[
D\cong E^{\mathrm{op}}.
\]
where \(E^{\mathrm{op}}\) denotes the opposite division algebra. This ambiguity reflects the fact that projective geometries arising from left \(E\)-modules and from right \(E\)-modules are naturally identified.
\end{theorem}

The theorem shows that sufficiently large projective ambiguity geometries
remember not only their dimension but also the underlying residue division
algebra. Consequently, terminal ambiguity geometry retains substantially more information in the nonsplit case: beyond the projective dimension, it can also encode the underlying residue division algebra.

\begin{proof}
The projective dimension of \(\mathbb P_D^{\,r-1}=\mathbb P_D(D^r)\) is
\(r-1\). Indeed, projective linear subspaces correspond to nonzero left
\(D\)-submodules of \(D^r\), and the longest strictly increasing chains of such
submodules have ranks \(1<2<\cdots<r\). Thus the projective dimension is
\(r-1\). Similarly, the projective dimension of \(\mathbb P_E^{\,s-1}\) is
\(s-1\). Since an isomorphism of projective geometries preserves projective linear
subspaces, it preserves projective dimension. Hence \(r-1=s-1\), and therefore
\(r=s\).

Since \(r,s\ge 3\), both projective geometries have projective dimension at
least \(2\). By the Fundamental Theorem of Projective Geometry, an isomorphism
of projective geometries of dimension at least \(2\) is induced by a semilinear
isomorphism between the underlying vector spaces. Thus the given isomorphism is induced by a semilinear isomorphism
\(D^r \longrightarrow E^r\). The semilinearity is with respect to an
isomorphism of the underlying division algebras, subject to the usual
left/right convention. Therefore \(D\cong E\), or, if one identifies projective
geometries built from opposite-sided modules, \(D\cong E^{\mathrm{op}}\). Hence \(D\) and \(E\) are recovered from the projective geometry up to the
standard opposite-algebra ambiguity.
\end{proof}

Applied to rank-one terminal ambiguity fibres satisfying \(D_j\cong k\), the theorem shows that sufficiently large projective ambiguity geometries remember not only their dimension but also the underlying residue division algebra \(D_i\), up to the standard opposite-algebra ambiguity.

Together with Corollary~\ref{cor:terminal-rank-recovery}, this gives reconstruction of the right rank and residue division algebra in the projective case. These results form the foundation for the species reconstruction statements below.

\section{Reconstruction of Terminal Species}

The reconstruction results of the previous section show that terminal
ambiguity geometry remembers both projective dimension and, in
sufficiently large rank, the underlying residue division algebra.
The next step is to determine how much of the terminal species can be
recovered from the ambiguity geometry itself.

We begin with a single terminal component and show that its ambiguity
fibre determines the corresponding local species data. We then combine
these local reconstruction results to recover the entire terminal
species.

The classification theorem expresses each terminal ambiguity fibre as a
projective geometry over the residue division algebra \(D_i\). The
preceding reconstruction results therefore suggest that the fibre should
already contain enough information to recover the corresponding local
species component.

\begin{corollary}\label{cor:local-species-reconstruction}
Let \(\Lambda\) be a finite-dimensional algebra with radical \(J\) satisfying \(J^n=0\). Let \(e_i,e_j\) be primitive idempotents and set \(D_i=e_i\Lambda e_i/e_iJe_i\), \(D_j=e_j\Lambda e_j/e_jJe_j\), \(V_{ij}=e_jJ^{n-1}e_i\), and \(r_{ij}=\dim_{D_i}V_{ij}\). Assume \(D_j\cong k\) and \(r_{ij}\ge3\). Then the projective geometry of \(\mathcal F^{<n,\mathrm{top},1}_{ij}\) determines the local terminal bimodule data \({}_{k}V_{ij}{}_{D_i}\), together with \(D_i\), up to the standard opposite-division-algebra ambiguity.
\end{corollary}

\begin{proof}
Since \(D_j\cong k\), Theorem~\ref{thm:grassmannian-classification}, specialized to \(s=1\), and Theorem~\ref{thm:nonsplit-projective-space} give a noncanonical identification
\[
\mathcal F^{<n,\mathrm{top},1}_{ij}
\cong
\mathbb P_{D_i}^{\,r_{ij}-1}.
\]
Since \(r_{ij}\ge3\), Theorem~\ref{thm:division-algebra-recovery} recovers \(D_i\), up to the standard opposite-division-algebra ambiguity, and the projective dimension recovers \(r_{ij}\). Since every finite-dimensional right \(D_i\)-module is free, \(V_{ij}\cong D_i^{\,r_{ij}}\) as a right \(D_i\)-module. The left action is the scalar \(k\)-action because \(D_j\cong k\). Hence the local bimodule data are determined up to the stated ambiguity.
\end{proof}

Thus, in rank at least three, the local ambiguity geometry contains the
full information of the corresponding terminal species component. The
only information lost is the standard opposite-division-algebra
ambiguity arising from projective geometry itself. For the purposes of the next result, we write
$\mathsf{Sp}^{\mathrm{top}}(\Lambda)=(D_i,V_{ij})_{i,j}$ for the terminal species of \(\Lambda\), where
$D_i=e_i\Lambda e_i/e_iJe_i$ and $V_{ij}=e_jJ^{n-1}e_i$.

\begin{theorem}\label{thm:global-species-reconstruction}
Let \(\Lambda\) be a finite-dimensional algebra with radical \(J\) satisfying $J^n=0$. Let $e_1,\ldots,e_t$
be a complete set of primitive idempotents. For each \(i,j\), set $D_i=e_i\Lambda e_i/e_iJe_i$, $D_j=e_j\Lambda e_j/e_jJe_j$, $V_{ij}=e_jJ^{n-1}e_i$, $r_{ij}=\dim_{D_i}V_{ij}$. Assume that $D_j\cong k$ whenever $V_{ij}\neq 0$, and that for every \(i\) such that
$J^{n-1}P_i\neq 0$, there exists at least one \(j\) with $r_{ij}\ge 3$. Then the labelled collection of terminal ambiguity fibres
\[
\left\{
\mathcal F^{<n,\mathrm{top},1}_{ij}
:
V_{ij}\neq 0
\right\}_{i,j}
\]
determines the terminal species
\[
\mathsf{Sp}^{\mathrm{top}}(\Lambda)
=
\bigl(D_i,{}_{D_j}V_{ij}{}_{D_i}\bigr)_{i,j}
\]
up to the standard opposite-division-algebra ambiguity for each recovered
residue division algebra \(D_i\).
\end{theorem}

\begin{proof}
Fix \(i\) such that \(J^{n-1}P_i\neq 0\). By hypothesis, there exists at least
one \(j\) such that \(r_{ij}\ge 3\). For this pair \((i,j)\), Corollary~\ref{cor:local-species-reconstruction} implies that the projective geometry of
\(\mathcal F^{<n,\mathrm{top},1}_{ij}\) determines \(D_i\) up to the standard
opposite-division-algebra ambiguity. Now fix any \(j\). If \(V_{ij}=0\), then, since the collection is labelled by the fixed vertex set, this zero component is determined by the absence of a fibre at \((i,j)\).

If \(V_{ij}\neq 0\), then \(D_j\cong k\) by hypothesis. Hence
Corollary~\ref{cor:terminal-rank-recovery} recovers \(r_{ij}\) from the
projective dimension of
\(\mathcal F^{<n,\mathrm{top},1}_{ij}\).
Since \(D_i\) has already been recovered from one rank-at-least-three
component, and since every finite-dimensional right \(D_i\)-module is free, the
module \(V_{ij}\) is determined up to isomorphism by
\(r_{ij}=\dim_{D_i}V_{ij}\). Indeed,
\(V_{ij}\cong D_i^{\,r_{ij}}\) as a right \(D_i\)-module. Since
\(D_j\cong k\), the left \(D_j\)-action is the scalar \(k\)-action. Thus, for
the fixed \(i\), the labelled family of fibres recovers all components
\(\bigl(D_i,{}_{D_j}V_{ij}{}_{D_i}\bigr)_j\).

Repeating this argument for every \(i\) with \(J^{n-1}P_i\neq 0\) recovers all
nonzero terminal species components. Components with \(J^{n-1}P_i=0\) have
\(V_{ij}=0\) for all \(j\), and hence contribute no nonzero terminal ambiguity
fibres. Therefore the labelled collection
\(\left\{
\mathcal F^{<n,\mathrm{top},1}_{ij}
:
V_{ij}\neq 0
\right\}_{i,j}\)
determines the terminal species
\(\mathsf{Sp}^{\mathrm{top}}(\Lambda)
=
\bigl(D_i,{}_{D_j}V_{ij}{}_{D_i}\bigr)_{i,j}\)
up to the standard opposite-division-algebra ambiguity for each recovered
residue division algebra \(D_i\).
\end{proof}

Under the hypotheses above, the theorem shows that terminal ambiguity geometry is rich enough to recover the complete terminal species. In particular, the terminal ambiguity fibres do not merely encode numerical invariants such as projective dimensions; they determine the residue division algebras and terminal module components from which the species is built.

This result forms the first half of the geometry and species equivalence developed in the following sections. The remaining task is to relate the reconstructed terminal species to the terminal radical layer and to determine how much information is retained when passing from the species to its associated ambiguity geometry.

\section{Reconstruction of the Terminal Radical Layer}

The previous section showed that the geometry of a terminal ambiguity
fibre recovers the local terminal species data from which it arises.
We now pass from individual fibres to the entire terminal radical
layer. Since the terminal species records precisely the components
\[
V_{ij}=e_jJ^{n-1}e_i,
\]
one expects the collection of terminal ambiguity fibres to retain
enough information to reconstruct the whole semisimple boundary
bimodule \(J^{n-1}\). The following theorem confirms that this is
indeed the case.

\begin{theorem}\label{thm:terminal-radical-layer-reconstruction} 
Let \(\Lambda\) be a finite-dimensional algebra with radical \(J\) satisfying $J^n=0$. Let $e_1,\ldots,e_t$ be a complete set of primitive idempotents. For each \(i,j\), set 
$D_i=e_i\Lambda e_i/e_iJe_i$, $D_j=e_j\Lambda e_j/e_jJe_j$, $V_{ij}=e_jJ^{n-1}e_i$, $r_{ij}=\dim_{D_i}V_{ij}$. Assume that $D_j\cong k$ whenever $V_{ij}\neq 0$, and that for every \(i\) with $J^{n-1}P_i\neq 0$, 
there exists at least one \(j\) such that $r_{ij}\ge 3$. Then the labelled terminal ambiguity geometry 
\[ 
\left\{ 
\mathcal F^{<n,\mathrm{top},1}_{ij} 
: 
V_{ij}\neq 0 
\right\}_{i,j} 
\] 
determines the terminal radical layer $J^{n-1}$ as a semisimple boundary bimodule over the recovered residue division 
algebras, up to the standard opposite-division-algebra ambiguity. 
\end{theorem} 
 
The argument is straightforward once the global species reconstruction 
theorem is available. The terminal radical layer is obtained by 
assembling the individual species components \(V_{ij}\), and these are 
precisely the objects recovered from the ambiguity geometry. 
 
\begin{proof} 
The terminal radical layer decomposes as 
\[
J^{n-1} 
= 
\bigoplus_{i,j} e_jJ^{n-1}e_i.
\] 
By definition,  $e_jJ^{n-1}e_i  =  V_{ij}$. By Theorem~\ref{thm:global-species-reconstruction}, the labelled collection 
of terminal ambiguity fibres determines the terminal species $\mathsf{Sp}^{\mathrm{top}}(\Lambda) = \bigl(D_i,{}_{D_j}V_{ij}{}_{D_i}\bigr)_{i,j}$ up to the standard opposite-division-algebra ambiguity for each recovered  residue division algebra \(D_i\). In particular, for every pair \((i,j)\), it determines the boundary component  ${}_{D_j}V_{ij}{}_{D_i}
={}_{D_j}(e_jJ^{n-1}e_i){}_{D_i}$ as part of the terminal species.  Since 
\[
J^{n-1} 
= 
\bigoplus_{i,j} V_{ij},
\] 
the collection of all recovered boundary components reconstructs 
\(J^{n-1}\) 
as a semisimple boundary bimodule over the recovered residue division 
algebras.  The only ambiguity is the standard opposite-division-algebra ambiguity already 
present in the projective-geometry reconstruction of the residue division 
algebras. Hence the labelled terminal ambiguity geometry determines 
\(J^{n-1}\) 
as a semisimple boundary bimodule, up to this standard ambiguity. 
\end{proof}

The theorem completes the reconstruction programme initiated in the
previous section. Starting from the labelled collection of terminal
ambiguity fibres, one recovers the terminal species and hence the
entire terminal radical layer as a semisimple boundary bimodule. Thus,
under the hypotheses above, terminal ambiguity geometry encodes not
merely numerical or local invariants, but the full semisimple boundary
structure carried by \(J^{n-1}\).

The natural remaining question is whether any information is lost in
passing between terminal species and terminal ambiguity geometry. The
next section compares these structures directly and proves that, under
the standing rank hypotheses, they are equivalent.

\section{Terminal Ambiguity Geometry and Terminal Species}

The previous sections established two complementary reconstruction
results. On the one hand, terminal ambiguity geometry determines the
terminal species under suitable rank hypotheses. On the other hand, the
classification theorem constructs the ambiguity geometry directly from
the terminal species. The purpose of this section is to make this
relationship precise. We first show that terminal ambiguity geometry
depends only on the terminal species. We then prove that, under the
standing rank assumptions, the two structures determine one another.
This yields a geometric realization of the terminal species and
identifies terminal ambiguity geometry as an equivalent encoding of the
terminal boundary data.

We begin with the easier direction. If two algebras have the same
terminal species, then their terminal ambiguity geometries agree. In
other words, the ambiguity geometry cannot detect information beyond
that already contained in the terminal species. For convenience, we
write
\[
\mathfrak A^{\mathrm{top}}(\Lambda)
=
\left\{
\mathcal F^{<n,\mathrm{top},s}_{ij}
:
V_{ij}\neq 0,\;
1\leq s\leq \ell_{ij}
\right\}_{i,j,s},
\qquad
\ell_{ij}=\dim_{D_j}V_{ij},
\]
for the labelled terminal ambiguity geometry of \(\Lambda\).

\begin{theorem}\label{thm:reconstruction-boundary} 
Let \(\Lambda\) and \(\Gamma\) be finite-dimensional algebras with radicals 
\(J_\Lambda\) and \(J_\Gamma\), satisfying $J_\Lambda^n=0$, $J_\Gamma^n=0$. Assume that their terminal species agree in the following labelled sense. 
There are primitive idempotents $e_1,\ldots,e_t\in \Lambda$ and $f_1,\ldots,f_t\in \Gamma$ 
such that, for every \(i\), there is an isomorphism of division algebras $\theta_i: D_i^\Lambda =  e_i\Lambda e_i/e_iJ_\Lambda e_i  \xrightarrow{\sim} f_i\Gamma f_i/f_iJ_\Gamma f_i 
= D_i^\Gamma$, and, for every pair \((i,j)\), there is an isomorphism ${}_{D_j^\Lambda}\!\left(e_jJ_\Lambda^{\,n-1}e_i\right){}_{D_i^\Lambda}
\cong {}_{D_j^\Gamma}\!\left(f_jJ_\Gamma^{\,n-1}f_i\right){}_{D_i^\Gamma}$ of \((D_j^\Lambda,D_i^\Lambda)\)-bimodules, where the bimodule structure on 
$f_jJ_\Gamma^{\,n-1}f_i$ is obtained by restriction of scalars along  \(\theta_j\) on the left and \(\theta_i\) on the right. Then the labelled terminal ambiguity geometries of \(\Lambda\) and \(\Gamma\) agree: $\mathfrak A^{\mathrm{top}}(\Lambda) \cong \mathfrak A^{\mathrm{top}}(\Gamma)$. In particular, terminal ambiguity geometry cannot distinguish finite-dimensional 
algebras with the same terminal species. 
\end{theorem} 
 
\begin{proof} 
For each pair \((i,j)\), set $V_{ij}^{\Lambda}=e_jJ_\Lambda^{\,n-1}e_i$ and $V_{ij}^{\Gamma}=f_jJ_\Gamma^{\,n-1}f_i$.
By hypothesis, after identifying 
\(D_i^\Lambda\cong D_i^\Gamma\) via \(\theta_i\) and
\(D_j^\Lambda\cong D_j^\Gamma\) via \(\theta_j\), one has 
\(V_{ij}^{\Lambda}\cong V_{ij}^{\Gamma}\) 
as bimodules over the same pair of division algebras. By 
Theorem~\ref{thm:grassmannian-classification}, for every admissible rank \(s\),
\[
\mathcal F^{<n,\mathrm{top},s}_{\Lambda,ij} 
\cong 
\operatorname{Gr}_{D_j^\Lambda}
\!\left(
\ell_{ij}-s,V_{ij}^{\Lambda}
\right)/(D_i^\Lambda)^\times,
\]
and 
\[
\mathcal F^{<n,\mathrm{top},s}_{\Gamma,ij} 
\cong 
\operatorname{Gr}_{D_j^\Gamma}
\!\left(
\ell_{ij}-s,V_{ij}^{\Gamma}
\right)/(D_i^\Gamma)^\times,
\]
where $\ell_{ij} =\dim_{D_j^\Lambda}V_{ij}^{\Lambda}=\dim_{D_j^\Gamma}V_{ij}^{\Gamma}$. Using the chosen isomorphisms 
\(\theta_i:D_i^\Lambda\xrightarrow{\sim}D_i^\Gamma\),
\(\theta_j:D_j^\Lambda\xrightarrow{\sim}D_j^\Gamma\),
and the corresponding bimodule isomorphism 
\(V_{ij}^{\Lambda}\cong V_{ij}^{\Gamma}\), the two Grassmannians are 
identified, and this identification intertwines the corresponding right
unit-group actions. Hence, for every pair \((i,j)\) and every admissible
rank \(s\),
\[
\mathcal F^{<n,\mathrm{top},s}_{\Lambda,ij} 
\cong 
\mathcal F^{<n,\mathrm{top},s}_{\Gamma,ij}.
\]
If \(V_{ij}^{\Lambda}=0\), then by the assumed species agreement also
\(V_{ij}^{\Gamma}=0\), so the pair \((i,j)\) contributes no terminal
ambiguity fibres to either labelled collection. Thus the zero and
nonzero components agree label by label. Therefore the labelled
collections of terminal ambiguity fibres agree: $\mathfrak A^{\mathrm{top}}(\Lambda) \cong
\mathfrak A^{\mathrm{top}}(\Gamma)$. Thus terminal ambiguity geometry depends only on the terminal species. In 
particular, it cannot distinguish algebras whose terminal species are 
isomorphic. 
\end{proof}

The theorem shows that the terminal species forms a reconstruction
boundary for terminal ambiguity geometry. Once the residue division
algebras and terminal module components are fixed, the associated
ambiguity geometry is completely determined. Thus every invariant arising from terminal ambiguity geometry factors
through the terminal species. The remaining question is whether any
information is lost in passing from the species to its geometry.

The reconstruction results obtained earlier suggest that, under suitable
rank hypotheses, no information is lost. Indeed, the ambiguity geometry
already recovers the residue division algebras and terminal module
components from which the species is assembled. The next theorem shows
that this reconstruction is complete.

\begin{theorem}\label{thm:geometry-species-equivalence} 
Let \(\Lambda\) be a finite-dimensional algebra with radical \(J\) satisfying $J^n=0$. Let $e_1,\ldots,e_t$ be a complete set of primitive idempotents. For each \(i,j\), set 
$D_i=e_i\Lambda e_i/e_iJe_i$, $D_j=e_j\Lambda e_j/e_jJe_j$, $V_{ij}=e_jJ^{n-1}e_i$, $r_{ij}=\dim_{D_i}V_{ij}$. Assume that $D_j\cong k$ whenever $V_{ij}\neq 0$, and that for every \(i\) with $J^{n-1}P_i\neq 0$, there exists at least one \(j\) such that $r_{ij}\ge 3$. 
Let
\[
\mathfrak A^{\mathrm{top},1}(\Lambda)
=
\left\{
\mathcal F^{<n,\mathrm{top},1}_{ij}
\right\}_{i,j}
\]
denote the labelled rank-one terminal ambiguity geometry, where $\mathcal F^{<n,\mathrm{top},1}_{ij}=\varnothing$ when $V_{ij}=0$. Then 
$\mathfrak A^{\mathrm{top},1}(\Lambda)$ determines and is determined by the terminal species
\[
\mathsf{Sp}^{\mathrm{top}}(\Lambda)
=
\bigl(D_i,{}_{D_j}V_{ij}{}_{D_i}\bigr)_{i,j},
\]
up to the standard opposite-division-algebra ambiguity for the recovered residue division algebras. Equivalently, under the stated hypotheses, the rank-one terminal
ambiguity geometry and the terminal species are mutually recoverable;
that is, the labelled rank-one terminal ambiguity geometry and the
terminal species determine one another.
\end{theorem} 
 
\begin{proof} 
First, the terminal species determines the terminal ambiguity geometry. The 
terminal species contains, for each \(i\), the residue division algebra 
\(D_i\) and, for each pair \((i,j)\), the bimodule ${}_{D_j}V_{ij}{}_{D_i}={}_{D_j}(e_jJ^{n-1}e_i){}_{D_i}$ If \(V_{ij}=0\), then the corresponding labelled 
ambiguity fibre is empty. If \(V_{ij}\neq 0\), then by hypothesis \(D_j\cong k\), so the left \(D_j\)-submodules of \(V_{ij}\) are precisely its \(k\)-subspaces. By Theorem~\ref{thm:grassmannian-classification}, specialized to \(s=1\),
\[
\mathcal F^{<n,\mathrm{top},1}_{ij}
\cong
\operatorname{Gr}_{k}
\bigl(\dim_kV_{ij}-1,V_{ij}\bigr)/D_i^\times.
\]
Identifying the Grassmannian of \(k\)-hyperplanes with
\(\mathbb P_k(V_{ij}^*)\), this becomes
\[
\mathcal F^{<n,\mathrm{top},1}_{ij}
\cong
\mathbb P_k(V_{ij}^*)/D_i^\times.
\]
By Theorem~\ref{thm:division-projectivization}, this quotient is naturally
identified with $\mathbb P_{D_i}\!\left(\operatorname{Hom}_{D_i}(V_{ij},D_i^*)\right)$,
where \(D_i^*=\operatorname{Hom}_k(D_i,k)\). Thus the terminal species
determines every labelled rank-one terminal ambiguity fibre, and hence determines
\(\mathfrak A^{\mathrm{top},1}(\Lambda)\). Conversely, the labelled rank-one terminal ambiguity geometry determines the terminal 
species under the stated hypotheses. Since the geometry is labelled by all 
pairs \((i,j)\), it records exactly which components are empty, hence which 
\(V_{ij}\) are zero. For each \(i\) with \(J^{n-1}P_i\neq 0\), the rank hypothesis provides some 
\(j\) with \(r_{ij}\ge 3\). For this pair, Theorem~\ref{thm:division-algebra-recovery} applied to 
$\mathcal F^{<n,\mathrm{top},1}_{ij} \cong \mathbb P_{D_i}^{\,r_{ij}-1}$ recovers \(D_i\) up to the standard opposite-division-algebra ambiguity. 
 
Now fix this \(i\). For every \(j\) with \(V_{ij}\neq 0\), the projective 
dimension of \(\mathcal F^{<n,\mathrm{top},1}_{ij}\) recovers 
\(r_{ij}=\dim_{D_i}V_{ij}\). Since every finite-dimensional right 
\(D_i\)-module is free, the pair \((D_i,r_{ij})\) determines the isomorphism 
class of \(V_{ij}\cong D_i^{\,r_{ij}}\) as a right \(D_i\)-module. Since
\(D_j\cong k\), the left \(D_j\)-action is the scalar \(k\)-action. Thus, for each \(i\), the labelled ambiguity geometry recovers all components 
$\bigl(D_i,{}_{D_j}V_{ij}{}_{D_i}\bigr)_j$. Repeating this for every \(i\) recovers the terminal 
species $\mathsf{Sp}^{\mathrm{top}}(\Lambda)=\bigl(D_i,{}_{D_j}V_{ij}{}_{D_i}\bigr)_{i,j}$, up to the standard opposite-division-algebra ambiguity. Combining the two directions, the labelled terminal ambiguity geometry and the terminal species determine one another.
\end{proof}

The theorem identifies the labelled rank-one terminal ambiguity geometry
and the terminal species as two descriptions of the same recoverable
terminal data, under the stated hypotheses. The species provides an
algebraic description in terms of residue division algebras and terminal
module components, while the rank-one ambiguity geometry provides a
projective-geometric realization of this data.

Consequently, under the hypotheses of the theorem, the terminal species
and the labelled rank-one terminal ambiguity geometry determine one
another, up to the stated opposite-division-algebra ambiguity. Thus the
algebraic data encoded by the terminal species can be recovered from the
rank-one ambiguity fibres, while those fibres are themselves determined
by the terminal species. This equivalence forms the conceptual culmination of the reconstruction
theory developed in the preceding sections and provides the foundation
for the study of automorphisms and symmetries in the next section.

\section{Symmetries of Terminal Ambiguity Geometry}

The rank-one projective description of terminal ambiguity also determines
the projective symmetries of the corresponding ambiguity fibres. In the
projective case considered here, these symmetries are governed by the
classical semilinear projective groups. Thus the geometry of terminal
ambiguity determines not only the points of the ambiguity fibre, but also
its intrinsic projective symmetry.

For a division algebra \(D\) and an integer \(r\ge3\), we denote by
\(P\Gamma L_r(D)\) the semilinear projective group of \(D^r\), equivalently
the automorphism group of the projective geometry
\(\mathbb P_D^{\,r-1}\).

\begin{theorem}
Let \(\Lambda\) be a finite-dimensional algebra with radical \(J\)
satisfying $J^n=0$. Let $D_i=e_i\Lambda e_i/e_iJe_i$,
$D_j=e_j\Lambda e_j/e_jJe_j$, $V_{ij}=e_jJ^{n-1}e_i$,
$r_{ij}=\dim_{D_i}V_{ij}$. Assume $D_j\cong k$ and $r_{ij}\ge 3$.
Set $W_{ij}=\operatorname{Hom}_{D_i}(V_{ij},D_i^*)$,
$D_i^*=\operatorname{Hom}_k(D_i,k)$. Then there is a canonical identification
\[
\operatorname{Aut}_{\mathrm{proj}}
\!\left(
\mathcal F^{<n,\mathrm{top},1}_{ij}
\right)
\cong
\operatorname{Aut}_{\mathrm{proj}}
\!\left(
\mathbb P_{D_i}(W_{ij})
\right).
\]
Moreover, after choosing a left \(D_i\)-module isomorphism
$W_{ij}\cong D_i^{\,r_{ij}}$, there is a noncanonical isomorphism
\[
\operatorname{Aut}_{\mathrm{proj}}
\!\left(
\mathcal F^{<n,\mathrm{top},1}_{ij}
\right)
\cong
P\Gamma L_{r_{ij}}(D_i),
\]
up to the standard left/right opposite-algebra convention.
\end{theorem}

\begin{proof}
Since $D_j\cong k$, Theorem~\ref{thm:grassmannian-classification},
specialized to \(s=1\), gives $\mathcal F^{<n,\mathrm{top},1}_{ij} \cong \mathbb P_k(V_{ij}^*)/D_i^\times$. By Theorem~\ref{thm:division-projectivization}, there is therefore a
canonical identification $\mathcal F^{<n,\mathrm{top},1}_{ij} \cong \mathbb P_{D_i}(W_{ij})$,
where
$W_{ij}=\operatorname{Hom}_{D_i}(V_{ij},D_i^*)$.
Therefore the projective automorphism group of the terminal ambiguity
fibre is canonically identified with the automorphism group of the
projective geometry \(\mathbb P_{D_i}(W_{ij})\). That is,
\[
\operatorname{Aut}_{\mathrm{proj}}
\!\left(
\mathcal F^{<n,\mathrm{top},1}_{ij}
\right)
\cong
\operatorname{Aut}_{\mathrm{proj}}
\!\left(
\mathbb P_{D_i}(W_{ij})
\right).
\]
Since \(r_{ij}\ge 3\), the projective geometry
\(\mathbb P_{D_i}(W_{ij})\) has projective dimension at least \(2\). By the
Fundamental Theorem of Projective Geometry, every projective automorphism of
\(\mathbb P_{D_i}(W_{ij})\) is induced by a semilinear automorphism of the
underlying left \(D_i\)-vector space \(W_{ij}\). Thus
\[
\operatorname{Aut}_{\mathrm{proj}}
\!\left(
\mathbb P_{D_i}(W_{ij})
\right)
\cong
P\Gamma L(W_{ij}).
\]

Finally, since \(V_{ij}\) is a finite-dimensional right \(D_i\)-module of rank
\(r_{ij}\), one has noncanonically
\[
W_{ij}
=
\operatorname{Hom}_{D_i}(V_{ij},D_i^*)
\cong
D_i^{\,r_{ij}}
\]
as a left \(D_i\)-module. Choosing such an isomorphism gives
\[
P\Gamma L(W_{ij})
\cong
P\Gamma L_{r_{ij}}(D_i).
\]
The choice is noncanonical, and the usual left/right convention may replace
\(D_i\) by \(D_i^{\mathrm{op}}\). Hence
\[
\operatorname{Aut}_{\mathrm{proj}}
\!\left(
\mathcal F^{<n,\mathrm{top},1}_{ij}
\right)
\cong
P\Gamma L_{r_{ij}}(D_i)
\]
noncanonically.
\end{proof}

Thus, in the rank-one projective case, the projective symmetries of a
terminal ambiguity fibre are completely described by classical projective
geometry. The ambiguity created by radical truncation therefore carries
not only a projective parameter space, but also the corresponding
semilinear projective symmetry.

\section{Conclusion}

The results of this paper develop a geometric approach to the ambiguity that remains after all proper radical truncations of a module have been fixed. This residual ambiguity is concentrated in the terminal radical layer \(J^{n-1}\). For primitive idempotents \(e_i,e_j\), the local terminal component \(V_{ij}=e_jJ^{n-1}e_i\) carries the natural bimodule structure \({}_{D_j}V_{ij}{}_{D_i}\), where \(D_h=e_h\Lambda e_h/e_hJe_h\). The left \(D_j\)-structure determines which terminal submodules are admissible, while the right \(D_i\)-structure determines when two admissible choices yield isomorphic quotient modules. This leads to the classification \(\mathcal F^{<n,\mathrm{top},s}_{ij}\cong\operatorname{Gr}_{D_j}(\ell_{ij}-s,V_{ij})/D_i^\times\), where \(\ell_{ij}=\dim_{D_j}V_{ij}\). Thus terminal ambiguity is organized into a hierarchy of Grassmannian orbit spaces indexed by terminal quotient rank, with the rank-one projective fibre appearing as the first case. These orbit spaces classify actual modules having identical proper radical truncation towers but different terminal completions. In particular, transitivity of the \(D_i^\times\)-action characterizes rigidity, while failure of transitivity gives nonisomorphic terminal completions with the same proper truncation data. The higher-rank classification therefore shows that the information left unresolved by proper radical truncation is not exhausted by the rank-one projective fibre.

In the split case, the classification becomes \(\mathcal F^{<n,\mathrm{top},s}_{ij}\cong\operatorname{Gr}_k(m_{ij}-s,V_{ij})\), where \(m_{ij}=\dim_kV_{ij}\). The dimension profile \(\delta_{ij}(s)=s(m_{ij}-s)\) measures the geometric size of the family of rank-\(s\) terminal completions left undetermined by the same proper radical truncation tower. It is maximal when \(s\) is as close as possible to \(m_{ij}/2\), with maximal dimension \(\lfloor m_{ij}^2/4\rfloor\), so higher terminal ranks can exhibit more ambiguity than the rank-one fibre. Over \(k=\mathbb F_q\), the cardinality \(\left|\mathcal F^{<n,\mathrm{top},s}_{ij}\right|={m_{ij}\brack s}_q\) gives the exact number of nonisomorphic rank-\(s\) terminal completions with the prescribed proper radical truncation tower. The significance of this geometry is module-theoretic. The terminal species records the division algebras and bimodules governing the terminal radical layer, while terminal ambiguity geometry describes how that boundary data is realized through families of modules that remain indistinguishable under every proper radical truncation. The Grassmannian classification therefore turns the information lost under radical truncation into a concrete module-completion problem and provides geometric criteria for rigidity and nonuniqueness together with geometric and enumerative measures of the resulting ambiguity.

The rank-one projective geometry also supports the reconstruction results developed in the later sections. Under the stated hypotheses, this geometry recovers the corresponding terminal data. In the projective case, the automorphisms of an individual terminal ambiguity fibre are governed by projective semilinear transformations. The higher-rank Grassmannian classification complements this reconstruction theory by describing the broader hierarchy of terminal completions left unresolved by proper radical truncation. The theory developed here is fundamentally terminal. The relations \(JJ^{n-1}=J^{n-1}J=0\) make \(J^{n-1}\) a semisimple \((\Lambda/J)\)-bimodule and reduce terminal completion to submodule geometry over residue division algebras. For lower radical layers this semisimplicity is generally absent, so analogous ambiguity problems would necessarily involve additional extension data. Understanding whether such lower-layer ambiguity admits natural geometric parameter spaces and corresponding reconstruction principles remains a direction for future investigation.

\bibliographystyle{amsplain}
\bibliography{references}

\end{document}